\documentclass[11pt]{amsart}

\usepackage{amsmath}
\usepackage{fullpage}
\usepackage{xspace}
\usepackage[psamsfonts]{amssymb}
\usepackage[latin1]{inputenc}
\usepackage{graphicx,color}
\usepackage[curve]{xypic} 
\usepackage{hyperref}
\usepackage{graphicx}

\usepackage{amsmath}%
\usepackage{amsthm}%
\usepackage{amscd}
\usepackage{amsfonts}%
\usepackage{amssymb}%
\usepackage{graphicx}

\usepackage{mathrsfs}

\usepackage{tikz}
\usetikzlibrary{matrix,arrows}

\usepackage{tikz-cd}

\newtheorem{theorem}{Theorem}[section]

\newtheorem{lemma}[theorem]{Lemma}

\newtheorem{proposition}[theorem]{Proposition}

\theoremstyle{remark}

\numberwithin{equation}{section}

\newcommand{\Mfrak}{\mathfrak{M}}

\newcommand{\Qfrak}{\mathfrak{Q}}

\newcommand{\Ccal}{\mathscr{C}}

\newcommand{\Ecal}{\mathscr{E}}

\newcommand{\Lcal}{\mathscr{L}}

\newcommand{\Pcal}{\mathscr{P}}

\newcommand{\Scal}{\mathscr{S}}

\newcommand{\Pro}{\mathbb{P}}

\newcommand{\Z}{\mathbb{Z}}

\newcommand{\Q}{\mathbb{Q}}
\newcommand{\R}{\mathbb{R}}
\newcommand{\N}{\mathbb{N}}
\newcommand{\A}{\mathbb{A}}

  \DeclareFontFamily{U}{wncy}{}
    \DeclareFontShape{U}{wncy}{m}{n}{<->wncyr10}{}
    \DeclareSymbolFont{mcy}{U}{wncy}{m}{n}
    \DeclareMathSymbol{\Sha}{\mathord}{mcy}{"58}

\begin{document}
\title[]{Effectivity for existence of rational points is undecidable}

\author{Natalia Garcia-Fritz}
\address{ Departamento de Matem\'aticas,
Pontificia Universidad Cat\'olica de Chile.
Facultad de Matem\'aticas,
4860 Av.\ Vicu\~na Mackenna,
Macul, RM, Chile}
\email[N. Garcia-Fritz]{natalia.garcia@uc.cl}%

\author{Hector Pasten}
\address{ Departamento de Matem\'aticas,
Pontificia Universidad Cat\'olica de Chile.
Facultad de Matem\'aticas,
4860 Av.\ Vicu\~na Mackenna,
Macul, RM, Chile}
\email[H. Pasten]{hector.pasten@uc.cl}%

\author{Xavier Vidaux}
\address{Universidad de Concepci\'on, Concepci\'on, Chile, Facultad de Ciencias F\'isicas y Matem\'aticas Departamento de Matem\'atica}
\email[X. Vidaux]{xvidaux@udec.cl}%

\thanks{N.G.-F. was supported by ANID Fondecyt Regular grant 1211004 from Chile. H.P was supported by ANID Fondecyt Regular grant 1230507 from Chile. X.V. was supported by ANID Fondecyt Regular grant 1210329 from Chile. The three authors were supported by the NSF Grant No. DMS-1928930 while at the MSRI, Berkeley, in the Summer of 2022.}
\date{\today}
\subjclass[2010]{Primary: 11U05; Secondary: 11C08, 11G50} %
\keywords{Hilbert's tenth problem, heights, effectivity, undecidable}%

\dedicatory{To the memory of Thanases Pheidas}

\begin{abstract} The analogue of Hilbert's tenth problem over $\mathbb{Q}$ asks for an algorithm to decide the existence of rational points in algebraic varieties over this field. This remains as one of the main open problems in the area of undecidability in number theory. Besides the existence of rational points, there is also considerable interest in the problem of effectivity: one asks whether the sought rational points satisfy determined height bounds, often expressed in terms of the height of the coefficients of the equations defining the algebraic varieties under consideration. We show that, in fact, Hilbert's tenth problem over $\mathbb{Q}$ with (finitely many) height comparison conditions is undecidable. 
\end{abstract}

\maketitle


\section{Introduction} 

\subsection{Hilbert's tenth problem} Hilbert's tenth problem asked for an algorithm to decide the existence of integral solutions to Diophantine equations over $\Z$. A negative solution was shown by Matiyasevich \cite{Mat} after the work of Davis, Putnam, and Robinson \cite{DPR}. Motivated by this result, the analogous problem has been investigated over various other rings and fields, and perhaps the most important open case is the field of rational numbers $\Q$. In this case, the required algorithm should decide the existence of rational points over any given algebraic variety over $\Q$.

Let us recall that the standard strategy to attack Hilbert's tenth problem for $\Q$ consists of trying to show that $\Z$ is Diophantine (i.e. positive existentially definable) in the field $\Q$, in order to transfer the undecidability of $\Z$ to $\Q$. Regarding this strategy, in 1949 J. Robinson \cite{Robinson} showed that $\Z$ is first order definable in $\Q$.  In 2007 Poonen \cite{PoonenZQ} showed that $\Z$ admits a definition of the form $\forall^2\exists^7$ in $\Q$ (meaning that the definition uses $2$ universal quantifiers followed by $7$ existential ones.) On the other hand, Koenigsmann \cite{KoeZQ} showed in 2010 that $\Z$ admits a $\forall\exists^n$-definition in $\Q$ (for certain $n$) and that $\Q-\Z$ is positive existentially definable in $\Q$. 

In a different direction, Poonen \cite{PoonenBig} showed that there is a computable set of prime numbers $S$ with density $1$ in the primes such that Hilbert's tenth problem for $S^{-1}\Z$ has a negative answer, see also \cite{Shl1, Shl2, Shl3}. This can be seen as an approximation to the case of $\Q$ in the sense that if one could take $S$ as the set of all prime numbers then $S^{-1}\Z$ would be equal to $\Q$. See \cite{EasyQ} and the references therein for further developments related to this approach.

See also \cite{Pheidas} for a detailed discussion on strategies to approach Hilbert's tenth problem over $\Q$.

It is often the case in Diophantine geometry that for a given algebraic variety over $\Q$ one not only asks about the existence of rational points, but there is also the question of \emph{effectivity}: one wants height bounds for the sought rational points. In this work we show that the problem of existence of solutions to Diophantine equations over $\Q$ with height conditions is undecidable. Before giving a precise formulation of our main results we need to introduce some notation.

\subsection{Heights}

Given a rational number $q=a/b$ with $a,b$ coprime integers, the (logarithmic) height of $q$ is $h(q)=\log \max\{|a|,|b|\}$. More generally, the height of a tuple ${\bf q}=(q_1,...,q_m)\in \Q^m$  is
$$
h_m({\bf q})=\log \max\{d,|dq_1|,...,|dq_m|\}
$$
where $d\ge 1$ is the least common denominator of the $q_j$. In particular, $h_1=h$.

For $m,n\ge 1$ let us define the height comparison relation $H_{m,n}$ on $\Q^m\times \Q^n$ by
$$
H_{m,n}({\bf x},{\bf y}):\quad h_m({\bf x})\le h_n({\bf y}), \quad \mbox{ for }{\bf x}\in \Q^m\mbox{ and } {\bf y}\in \Q^n.
$$

These relations $H_{m,n}$ allows us to formulate questions of effectivity for the existence of rational points in varieties. For instance, given an integer $M\ge 1$ consider the existential statement
$$
\begin{aligned}
\Pcal_M: &\quad \mbox{There are }(x_1,x_2)\in \Q^2 \mbox{ and }a\in \Q-\{-1,0,1\}\mbox{ such that }\\
&\qquad x_2^2=x_1^5+a\mbox{ and }H_{1,2}(a^M, (x_1,x_2)).
\end{aligned}
$$
Note that if $\Pcal_M$ fails for certain fixed $M$, then ``effective Mordell'' holds for the family of genus 2 hyperelliptic curves $x_2^2=x_1^5+a$ in the form $h_2(x_1,x_2) < M\cdot h(a)$. (As usual, by ``effective Mordell'' one means Faltings's theorem \cite{Faltings1} with a bound for the height of the rational points of a curve in terms of the height of the coefficients of an equation for the curve.) While a height bound as the previous one is expected for some $M$,  effective Mordell is a major open problem even for the family of hyperelliptic curves mentioned above ---see \cite{Alpoge} for a concrete example where a form of effective Mordell holds: namely, for the $1$-parameter family of curves $x^6+4y^3=a^2$ where one can give a computable bound for the height of rational points in terms of $a$, although it is not polynomial on the (multiplicative) height of the parameter $a$.

\subsection{Undecidability}

 Let $\Lcal_h$ be the language formed by the constant symbols $0,1$, the function symbols $+$ and $\times$, and the relation symbols $=$ and $H_{m,n}$ for $m,n\ge 1$; that is, $\Lcal_h$ is the language of arithmetic endowed with the symbols $H_{m,n}$. Let us interpret the symbols of $\Lcal_h$ on $\Q$ in the obvious way, thus obtaining an $\Lcal_h$-structure that we denote by $\Qfrak$.

The positive existential theory of the structure $\Qfrak$ precisely formalizes the problem of existence of rational points in varieties with height inequalities conditions. For instance, the assertions $\Pcal_M$ can be restated as the following positive existential sentence over the language $\Lcal_h$:
$$
\exists x_1\exists x_2 \exists a (a^3\ne a)\wedge (x_2^2=x_1^5+a)\wedge H_{1,2}(a^M,(x_1,x_2))
$$
where we use the fact that $\ne$ is positive existentially definable over any field in the language of rings.

With these notation, we can formulate our main results.
\begin{theorem}[Interpretability] \label{ThmMain1} There is a positive existential interpretation of the semiring $(\N;0,1,+,\times,=)$ in the structure $\Qfrak$. 
\end{theorem}
Here we are using the notation $\N=\{0,1,2,...\}$. Since the negative solution of Hilbert's tenth problem gives that the positive existential theory of $(\N;0,1,+,\times,=)$ is undecidable,  we deduce:
\begin{theorem}[Undecidability] \label{ThmMain2}
The positive existential theory of $\Q$ over the language $\Lcal_{h}$ is undecidable. Thus, there is no algorithm that takes as input a system of Diophantine equations over $\Q$ with height comparison conditions $H_{m,n}$ and decides whether or not there is a rational solution satisfying the prescribed height conditions.
\end{theorem}

Theorem \ref{ThmMain2} is a formal consequence of Theorem \ref{ThmMain1}; see Proposition \ref{PropUnd} below.  Let us also mention that, as the proof will show, one only needs the relations $H_{m,n}$ for $m,n\le 3$ to prove Theorems \ref{ThmMain1} and Theorem \ref{ThmMain2}.

The proof of Theorem \ref{ThmMain1} uses the theory of heights on elliptic curves to positive existentially interpret the structure $(\N;0,1,+,B(a,b),=)$ in $\Qfrak$, where $B(a,b)$ is the binary relation stating that $a$ and $b$ are consecutive squares; that is, there is $n$ with $a=n^2$ and $b=(n+1)^2$. This suffices, since it is easy to see that in this structure the multiplication function is positive existentially definable, see Lemma \ref{LemmaMult}.


\section{Background} 

\subsection{Interpretations} In this paragraph we recall the notion of interpretation. See for instance \cite{Hodges} for further details.

Let $\Lcal_1$ and $\Lcal_2$ be first order languages and let $\Mfrak_1$ and $\Mfrak_2$ be structures over these languages respectively, with base sets $M_1$ and $M_2$. An interpretation of rank $r$ of the structure $\Mfrak_1$ in $\Mfrak_2$ is the following data:
\begin{itemize}
\item[(i)] a formula $\phi_{\Lcal_1}$ over $\Lcal_2$ with $r$ free variables;
\item[(ii)] a function $\theta:\phi_{\Lcal_1}^{\Mfrak_2}\to M_1$ (the upper script $\Mfrak_2$ stands for the realization of a formula, as usual);
\item[(iii)] for each $s\in \Lcal_1$, a formula $\phi_s$ over $\Lcal_2$
\end{itemize}
satisfying the following conditions:
\begin{itemize}
\item[(a)] the function $\theta$ is surjective onto $M_1$;
\item[(b)] $\phi_c^{\Mfrak_2}\subseteq \phi_{\Lcal_1}^{\Mfrak_2}$ and $\theta^{-1}(c^{\Mfrak_1})=\phi_c^{\Mfrak_2}$ for each symbol of constant $c\in \Lcal_1$;
\item[(c)] $\phi_R^{\Mfrak_2}\subseteq (\phi_{\Lcal_1}^{\Mfrak_2})^n$ and $(\theta^n)^{-1}(R^{\Mfrak_1})=\phi_R^{\Mfrak_2}$ for each symbol of $n$-ary relation $R\in \Lcal_1$;
\item[(d)] $\phi_f^{\Mfrak_2}\subseteq (\phi_{\Lcal_1}^{\Mfrak_2})^{n+1}$ and $(\theta^{n+1})^{-1}(f^{\Mfrak_1})=\phi_f^{\Mfrak_2}$ for each symbol of $n$-ary function $f\in \Lcal_1$.
\end{itemize}
Here, $\theta^k:(\phi_{\Lcal_1}^{\Mfrak_2})^k\to M_1^k$ stands for the map obtained from $k$ copies of $\theta$.

If all the formulas $\phi_s$ for $s\in \{\Lcal_1\}\cup\Lcal_1$ are positive existential, we say that the interpretation is positive existential. For simplicity, the interpretation will be simply called $\theta$ if no confusion can occur. 

The relevance of interpretations for our purposes is the following standard application of interpretability:

\begin{proposition}[Transference of undecidability by interpretations]\label{PropUnd} Suppose that there is an interpretation of $\Mfrak_1$ in $\Mfrak_2$. If the first order theory of $\Mfrak_1$ is undecidable, then so is the first order theory of $\Mfrak_2$. Furthermore, if the interpretation is positive existential and if the positive existential theory of $\Mfrak_1$ is undecidable, then the positive existential theory of $\Mfrak_2$ is undecidable.
\end{proposition}

\subsection{Elliptic curves} We need a couple of facts regarding heights on elliptic curves. These results are standard and can be found, for instance, in \cite{Silverman}.

Let $\Ecal$ be an elliptic curve over $\Q$ given by a fixed Weierstrass equation. Then the point at infinity $P_0$ can be taken as the neutral element of the group of rational points $\Ecal(\Q)$, and the $x$-coordinate defines a morphism $\pi: \Ecal\to \Pro^1$ that maps $\Ecal-\{P_0\}$ to the affine line $\A^1$.

For $P\in \Ecal(\Q)$ define the naive height
$$
h_\pi(P)=\begin{cases}
0&\mbox{ if }P=P_0\\
h(\pi(P))&\mbox{ if }P\ne P_0.
\end{cases}
$$
Given an integer $n$ the multiplication-by-$n$ map on $\Ecal$ is denoted by $[n]$. The limit
$$
\hat{h}(P):=\lim_{k\to \infty} \frac{h_\pi([2^k]P)}{4^k}
$$
exists and it defines a function $\hat{h}:\Ecal(\Q)\to \R_{\ge 0}$ called the \emph{canonical height} (some authors normalize it by dividing by $2$.) 

\begin{proposition}[Properties of the canonical height]\label{LemmaBasicE} The canonical height has the following properties:
\begin{itemize}
\item[(i)] There is a constant $c_\Ecal>0$ depending only on $\Ecal$ such that for all $P\in \Ecal(\Q)$ we have
$$
|\hat{h}(P)-h_\pi(P)|< c_\Ecal.
$$
\item[(ii)] For every $n\in \Z$ and $P\in \Ecal(\Q)$ we have $\hat{h}([n]P)=n^2\hat{h}(P)$.
\item[(iii)] For $P\in \Ecal(\Q)$, we have $\hat{h}(P)=0$ if and only if $P$ is a torsion point. 
\end{itemize}
\end{proposition}


\section{Preliminary results}

\subsection{Sum of heights} 

We will need the following elementary height formula:

\begin{lemma}[Addition of heights]\label{LemmaHsum} Let $q_1,q_2\in \Q$. Then
$$
h(q_1)+h(q_2)=h_3(q_1,q_2,q_1q_2).
$$
\end{lemma}

For the proof it will be convenient to have an alternative expression for the height. For a place $v$ of $\Q$ we let $|\cdot |_v$ be the standard absolute value attached to $v$ (in the $p$-adic case, normalized by $|p|_p=1/p$.) For $q\in \Q^\times$ the product formula yields
$$
\sum_v \log |q|_v =0.
$$
Then for $q_1,...,q_n\in \Q$ one has
$$
h_n(q_1,...,q_n)=\sum_v \log \max\{1,|q_1|_v,...,|q_n|_v\}.
$$
In fact, this is an easy consequence of the product formula: one recovers the formula defining $h_n$ upon adding $\sum_v \log |d|_v=0$ where $d$ is the least common denominator of the $q_j$. 

With this at hand we can prove the previous lemma.

\begin{proof}[Proof of Lemma \ref{LemmaHsum}] Decomposing the height over all places $v$ of $\Q$ we get
$$
\begin{aligned}
h_3(q_1,q_2,q_1q_2) &= \sum_v \log \max\{1,|q_1|_v,|q_2|_v,|q_1q_2|_v\}\\
&=\sum_v \log\left(\max\{1,|q_1|_v\}\cdot \max\{1,|q_2|_v\}\right)\\
&=h(q_1)+h(q_2).
\end{aligned}
$$
\end{proof}

\subsection{Some height equalities} Recall that the structure $\Qfrak$ consists of the field $\Q$ over the language of rings expanded by the relation symbols $H_{m,n}$ for height comparisons; this expanded language is denoted by $\Lcal_h$. We need to positive existentially define some additional height relations in this structure.

First we have equality of heights. We define the relation $E_{m,n}$ on $\Q^m\times \Q^n$ as follows:
$$
E_{m,n}({\bf x},{\bf y}):\quad h_m({\bf x})= h_n({\bf y}), \quad \mbox{ for }{\bf x}\in \Q^m\mbox{ and } {\bf y}\in \Q^n.
$$

\begin{lemma}[Equality of heights]\label{LemmaEquality} The relations $E_{m,n}$ are positive existentially definable in the $\Lcal_h$-structure $\Qfrak$. 
\end{lemma}
\begin{proof} Indeed, $E_{m,n}({\bf x},{\bf y})$ is defined by the quantifier-free formula $H_{m,n}({\bf x},{\bf y})\wedge H_{n,m}({\bf y},{\bf x})$.
\end{proof}

Using this and Lemma \ref{LemmaHsum}, we deduce:

\begin{lemma}[Defining the sum of heights]\label{LemmaDefSum} The ternary relation $S(x,y,z)$ on $\Q^3$ defined by
$$
h(x)+h(y)=h(z)
$$
is positive existentially $\Lcal_h$-definable.
\end{lemma}
\begin{proof} By Lemma \ref{LemmaEquality} the expression $E_{1,3}(z,(x,y,xy))$ gives a positive existentially $\Lcal_h$-definable ternary relation. We note that it defines $S(x,y,z)$; indeed, by Lemma \ref{LemmaHsum} we have $h_3(x,y,xy)=h(x)+h(y)$ and the claim follows.
\end{proof}

\subsection{Approximate height relations} We need to allow comparison of heights with bounded errors. For this, first we have:

\begin{lemma}[Approximate height comparison, version 1]\label{LemmaApprox1} There is a positive existentially $\Lcal_h$-definable binary relation $A(x,y)$ on $\Q^2$ with the following properties: given $(x,y)\in \Q^2$ we have
\begin{itemize}
\item[(i)] if  $h(x)\le h(y)+1$, then $A(x,y)$ holds; and
\item[(ii)] if $A(x,y)$ holds, then $h(x)\le h(y)+2$.
\end{itemize}
\end{lemma}
\begin{proof}
We first note that the set $J=\{q\in \Q : 1\le q\le 2\}$ has the property that for every $r\in \Q$ there is $q\in J$ with $h(r)=h(q)$. Indeed, $h(1)=0$ and for every integer $n\ge 2$ we have $h(n/(n-1))=\log n$. Furthermore, by Lagrange's four squares theorem, the set $J$ is positive existentially definable in $\Q$ over the language of rings.

We claim that for all $q\in J$ we have
$$
h(q)+\log(7/2)\le h(q+5)\le h(q)+\log(6).
$$
Indeed, write $q=a/b$ with $a\ge b$ coprime positive integers so that $h(q)=\log a$. Since $q+5=(a+5b)/b$ with $\gcd(b,a+5b)=1$, we have $h(q+5)=\log(a+5b)$. As $q\in J$, we have $a/2\le b\le a$ which gives $7a/2\le a+5b\le 6a$, and the claim follows.

In particular, for all $q\in J$ we get
$$
h(q)+1< h(q+5)< h(q)+2.
$$
Finally, we take $A(x,y)$ as the relation defined by the following expression, which is positive existentially $\Lcal_h$-definable:
$$
\exists q,  (q\in J)\wedge E_{1,1}(y,q)\wedge H_{1,1}(x,q+5).
$$
\end{proof}

\begin{lemma}[Approximate height comparison, version 2]\label{LemmaApprox2} Let $M\ge 1$ be an integer. There is a positive existentially $\Lcal_h$-definable binary relation $A^M(x,y)$ on $\Q^2$ with the following properties: given $(x,y)\in \Q^2$ we have
\begin{itemize}
\item[(i)] if  $h(x)\le h(y)+M$, then $A^M(x,y)$ holds; and
\item[(ii)] if $A^M(x,y)$ holds, then $h(x)\le h(y)+4M$.
\end{itemize}
\end{lemma}
\begin{proof} We let $A^1=A$ with $A$ the relation provided by Lemma \ref{LemmaApprox1}. 

For $M\ge 2$ we can take $A^M(x,y)$ as
$$
\exists t_1,...t_{2M-1}, A(x,t_1)\wedge A(t_1,t_2)\wedge\cdots \wedge A(t_{2M-1},y).
$$

The required item (ii) is clear from the properties of $A(x,y)$. 

For (i), suppose that $h(x)\le h(y)+M$. Let $u,v$ be positive integers with $h(x)=\log (u)$ and $h(y)=\log (v)$, so that $\log(u)\le \log(v) +M$, that is, $u\le e^Mv$. For $j=1,2,...,2M-1$ choose $t_j=2^{2M-j}v$. Then:
\begin{itemize}
\item Note that $u\le e^Mv< 4^Mv=2t_1$. So, $\log(u)< \log(t_1)+\log 2<\log (t_1) +1$ and $A(u,t_1)$ holds. This means that $A(x,t_1)$ holds.
\item For $j=1,2,...,2M-2$ we have $t_j=2t_{j+1}$. Hence, $\log(t_j)=\log(t_{j+1})+\log 2< \log (t_{j+1})+1$ and $A(t_{j},t_{j+1})$ holds.
\item Finally, $t_{2M-1}=2v$ and we get $\log(t_{2M-1})=\log(v)+\log 2< \log (v)+1$. Therefore $A(t_{2M-1},v)$ holds. This means that $A(t_{2M-1},y)$ holds.
\end{itemize}
In this way we see that $A^M(x,y)$ holds, by choosing $t_j$ as above.
\end{proof}

The previous lemma allows us to positive existentially define approximate equality of heights.

\begin{lemma}[Approximate equality of heights]\label{LemmaApproxE} Let $M\ge 1$ be an integer. There is a positive existentially $\Lcal_h$-definable binary relation $E^M(x,y)$ on $\Q^2$ with the following properties: given $(x,y)\in \Q^2$ we have
\begin{itemize}
\item[(i)] if $|h(x)-h(y)|\le M$ then $E^M(x,y)$ holds; and
\item[(ii)] if $E^M(x,y)$ holds, then $|h(x)-h(y)|\le 4M$.
\end{itemize}

\end{lemma}
\begin{proof}
The required relation can be defined by the expression
$$
A^M(x,y)\wedge A^M(y,x)
$$ 
which, by Lemma \ref{LemmaApprox2}, is positive existentially $\Lcal_h$-definable. 
\end{proof}

Finally, we need a definition for strict inequality of heights. We only need it in a very particular case, which is what we prove now:

\begin{lemma}[Strict inequality of heights]\label{LemmaLess} There is a positive existentially $\Lcal_h$-definable binary relation $L(x,y)$ on $\Q^2$ with the following properties: given $(x,y)\in \Q^2$ we have
\begin{itemize}
\item[(i)] if $h(x)+11\le h(y)$ then $L(x,y)$ holds; and
\item[(ii)] if $L(x,y)$ holds, then $h(x)+10\le h(y)$.
\end{itemize}
\end{lemma}
\begin{proof} Let $J=\{q\in \Q:1\le q\le 2\}$. With an argument analogous to that in the proof of Lemma \ref{LemmaApprox1}, one sees that the formula
$$
\exists q,(q\in J)\wedge E_{1,1}(x,q)\wedge H_{1,1}(q+50000,y)
$$
works. Details are left to the reader. (The numerical choice is due to the fact that $\log(50000/2+1)>10$ and $\log(50000+1)<11$.)
\end{proof}


\section{Interpretation of the integers}

\subsection{An elliptic curve}\label{SecEllChoice} For the proof of Theorem \ref{ThmMain1} we need an elliptic curve $\Ecal$ over $\Q$ whose Mordell--Weil group is $\Ecal(\Q)\simeq \Z$. There are plenty of such curves and any choice would work for the argument. However, to keep the presentation simple and concrete, let us choose one explicit example: from now on $\Ecal$ denotes the elliptic curve
$$
\Ecal:\quad y^2=x^3+2.
$$
This is the elliptic curve 1728.n4 in the LMFDB \cite{LMFDB}. From this database we see that the point $P_1=(-1,1)$ generates the Mordell--Weil group $\Ecal(\Q)$, which is isomorphic to $\Z$. Furthermore, also from this database we see that the canonical height of $P_1$ is 
$$
\hat{h}(P_1)=0.7545769...
$$
From Example 2.1 in \cite{SilvermancE} with $B=2$ (note that in this reference $\hat{h}$ is normalized by $2$) we see that every point $P\in \Ecal(\Q)$ satisfies 
$$
-3.192< \hat{h}(P) - h_\pi(P)< 3.384.
$$
In particular, the constant $c_\Ecal$ from Proposition \ref{LemmaBasicE} can be taken as  $c_\Ecal=4$.

 For later reference let us define $D=200^2\hat{h}(P_1)\in \R_{>0}$ and note that $D>30000$ (these numerical choices are not optimal, but they are more than enough for our purposes.) Let $\Gamma=[200]\Ecal (\Q)$, then $\Gamma\simeq \Z$ is generated by $Q_1:=[200]P_1$ which has canonical height $D$.  The elements of $\Gamma$ are precisely the points $Q_k:=[k]Q_1=[200k]P_1$ for $k\in \Z$ (note that $Q_0=P_0$ is the point at infinity) and we observe that 
$$
\hat{h}(Q_k)=k^2D\quad \mbox{ for each }k\in \Z.
$$

\subsection{The interpretation function}\label{SecTheta}

We let 
$$
X=\{\pi(Q) : Q_0\ne Q\in \Gamma\}\cup \{0\}\subseteq \Q.
$$
Note that by definition of $\Gamma$, this set $X$ is positive existentially definable in $\Q$ over the language of rings. This is because the map $[200]:\Ecal\to \Ecal$ can be expressed in coordinates by a fixed rational function and $\Gamma=[200]\Ecal(\Q)$. Next we define
$$
X_1 = \{q\in \Q : \exists \gamma\in X, h(q)=h(\gamma)\}
$$
and for $n\ge 1$ we let
$$
X_{n+1} = \{q\in \Q : \exists u\in X_n,\exists v\in X_1, h(u)+h(v)=h(q) \}.
$$
\begin{lemma}\label{LemmaDefXn} The sets $X_n\subseteq \Q$ are positive existentially $\Lcal_h$-definable.
\end{lemma}
\begin{proof} These sets are inductively defined using the relations $E_{1,1}(x,y)$ and $S(x,y,z)$. Both of these relations are positive existentially $\Lcal_h$-definable by Lemmas \ref{LemmaEquality} and \ref{LemmaDefSum}.
\end{proof}

The following lemma shows that the height of the elements of $X_4$ is controlled up to a bounded error term.

\begin{lemma}\label{LemmaErrorX4} For every $q\in X_4$ there is a unique non-negative integer $m_q\in \N$ such that  
$$
|h(q)-m_qD|\le 16.
$$
Furthermore, we have:
\begin{itemize}
\item[(i)] if $q\in X_4$ and $m\in \Z$ satisfies $|h(q)-mD|\le 15000$ then $m=m_q$, and
\item[(ii)] for each non-negative integer $m\in \N$ there is $q\in X_4$ with $m=m_q$.
\end{itemize}
\end{lemma}
\begin{proof} Let $q\in X_4$. Then there are $u_1,u_2,u_3,u_4\in X_1$ such that 
$$
h(q)=\sum_{j=1}^4 h(u_j).
$$
Since $u_j\in X_1$, there are $\gamma_j\in X=\{0\}\cup \pi(\Gamma-\{Q_0\})\subseteq \Q$ with $h(u_j)=h(\gamma_j)$. Thus, there are integers $k_1,k_2,k_3,k_4\in \Z$ such that 
$$
h_\pi(Q_{k_j})=h(u_j)\quad \mbox{ for }j=1,2,3,4.
$$
(In the case $\gamma_j=0$ we choose $k_j=0$ since $h_\pi(Q_0)=0=h(0)$.)  By the defining property of the constant $c_\Ecal$ (which in our case can be taken as $c_\Ecal=4$)  we get
$$
4> |h_\pi(Q_{k_j})- \hat{h}(Q_{k_j})| = |h_\pi(Q_{k_j})- k_j^2D|.
$$
Therefore, we have
$$
|h(q)-(k_1^2+k_2^2+k_3^2+k_4^2)D|\le  16.
$$
Let us take $m_q= k_1^2+k_2^2+k_3^2+k_4^2$. The uniqueness of $m_q$ and item (i) follow from the fact that $D>30000$, so the numbers of the form $mD\in \R$ for $m\in \Z$ are separated by a distance of at least $30000$. 

Finally, let us show that every $m\in \N$ occurs as $m_q$ for some $q\in X_4$ (this is item (ii)). Given $m$, by Lagrange's four squares theorem we can write $m=k_1^2+k_2^2+k_3^2+k_4^2$ for some integers $k_j$. Let us choose $u_j=\pi(Q_{k_j})\in X_1$ if $k_j\ne 0$ and $u_j=0\in X_1$ if $k_j=0$. Let $r_j\ge 1$ be integers such that $h(u_j)=\log r_j$ and let $q=r_1r_2r_3r_4$. Then we have 
$$
h(q)=\log (r_1r_2r_3r_4)=\sum_{j=1}^4 h(u_j),
$$
which shows that $q\in X_4$. With this choice of $q$, the same computation as above gives  $m_q=m$.
\end{proof}

For $q\in X_4$ let $\theta(q)=m_q$ be the integer afforded by the previous lemma. We get a well-defined function
$$
\theta: X_4\to \N.
$$
Let us reformulate Lemma \ref{LemmaErrorX4} in terms of $\theta$:
\begin{lemma}\label{LemmaTheta} The function $\theta:X_4\to \N$ is surjective. Furthermore,  for every $q\in X_4$ we have
$$
|h(q)-\theta(q)D|\le 16.
$$
In addition, if for some $q\in X_4$ and  $m\in \Z$ we have $|h(q)-mD|\le  15000$, then $m=\theta(q)$.
\end{lemma}

\subsection{An auxiliary structure}\label{SecAuxiliary}

The interpretation of the semiring $\N$ in $\Qfrak$ will be constructed using the function $\theta$ defined above. For this, it will be convenient to introduce an auxiliary structure.

Let $B(x,y)$ be the binary relation on $\N$ defined by
$$
B(x,y):\quad \mbox{ there is $k\ge 0$ such that $x=k^2$ and $y=(k+1)^2$.}
$$
We consider the language $\Lcal_B=\{0,1,+,B,=\}$. Note that $\N$ is an $\Lcal_B$-structure in a natural way. The next lemma is standard; we include the proof for the sake of completeness.

\begin{lemma}\label{LemmaMult} Multiplication in $\N$ is positive existentially $\Lcal_B$-definable. 
\end{lemma}
\begin{proof} Define the binary relation $\sigma(x,y)$ on $\N$ by
$$
\sigma(x,y):\quad y=x^2,
$$
which is the graph of the squaring function. We claim that $\sigma$ is positive existentially $\Lcal_B$-definable. Indeed, one sees that the formula
$$
\exists z, B(y,z)\wedge(z=y+2x+1)
$$
gives a positive existential $\Lcal_B$-definition for $\sigma(x,y)$.
Finally, we note that the relation $x\cdot y=z$ is defined by the expression
$$
\exists u\exists v\exists w, \sigma(x,u)\wedge \sigma(y,v)\wedge\sigma(x+y,w)\wedge(w=u+2z+v).
$$
\end{proof}

\subsection{Definability in the rationals with heights}

In order to prove Theorem \ref{ThmMain1} it suffices to interpret the $\Lcal_B$-structure given by $\N$ in the $\Lcal_h$-structure $\Qfrak$. For this we use the function $\theta$ constructed in Section \ref{SecTheta}.
\begin{lemma}\label{LemmaDef1} The sets $\theta^{-1}(0)\subseteq X_4$, $\theta^{-1}(1)\subseteq X_4$, and $(\theta^2)^{-1}(=)\subseteq X_4^2$ are positive existentially $\Lcal_h$-definable.
\end{lemma}
\begin{proof} Recall the approximate equality relation $E^{M}(x,y)$ from Lemma \ref{LemmaApproxE}. We will repeatedly use the properties of $E^M(x,y)$ as well as Lemma \ref{LemmaTheta} without mention.

First we claim that the expression $E^{16}(0,x)\wedge (x\in X_4)$ (which is positive existentially $\Lcal_h$-definable) defines $\theta^{-1}(0)$. Indeed, if $E^{16}(0,x)$ holds then $|h(x)-0|\le 4\cdot 16=64$ and since $x\in X_4$ we have that $\theta(x)$ must be $0$. Conversely, if $x\in X_4$ and $\theta(x)=0$, then $|h(x)-0\cdot D|\le  16$, which implies that $E^{16}(0,x)$ holds. This proves the result for $\theta^{-1}(0)\subseteq X_4$.

Let $q_1\in X_4$ be a rational number with $\theta(q_1)=1$ (this is possible as $\theta$ is surjective.) We claim that the expression $E^{32}(q_1,x)\wedge (x\in X_4)$ defines $\theta^{-1}(1)$. Indeed, if $E^{32}(q_1,x)$ holds then 
$$
|h(x)-1\cdot D|\le |h(x)-h(q_1)|+|h(q_1)-D|\le  4\cdot 32 + 16=144. 
$$
We conclude $\theta(x)=1$. On the other hand, if $x\in \theta^{-1}(1)$ we have
$$
 |h(x)-h(q_1)|\le |h(x)-D| + |h(q_1)-D|\le  16+16=32,
$$
which implies that $E^{32}(q_1,x)$ holds.

Finally, we claim that the expression $E^{32}(x,y)\wedge (x,y\in X_4)$ defines $(\theta^2)^{-1}(=)$. Indeed, if $x,y\in X_4$ and $E^{32}(x,y)$ holds, then
$$
|\theta(x)-\theta(y)|\cdot D \le |h(x)-\theta(x)D|+|h(y)-\theta(y)D|+|h(x)-h(y)|\le 16+16+4\cdot 32=160
$$
which implies $\theta(x)=\theta(y)$ since $D>30000$. Conversely, if $x,y\in X_4$ and $\theta(x)=\theta(y)=m$ then
$$
|h(x)-h(y)|\le |h(x)-mD|+|h(y)-mD|\le  32,
$$
implying that $E^{32}(x,y)$ holds.
\end{proof}

\begin{lemma}\label{LemmaDef2} Let $\Scal=\{(x,y,z)\in \N^3 : x+y=z\}\subseteq \N^3$. Then $(\theta^3)^{-1}(\Scal)\subseteq X_4^3$ is positive existentially $\Lcal_h$-definable.
\end{lemma}
\begin{proof} Let us recall the relation $S(x,y,z)$ on $\Q^3$ from Lemma \ref{LemmaDefSum}. We claim that the expression
$$
S'(x,y,z):\quad (x,y,z\in X_4)\wedge \exists w\in \Q, E^{48}(w,z)\wedge S(x,y,w)
$$
defines $(\theta^3)^{-1}(\Scal)$. This suffices, as $S'$ is positive existentially $\Lcal_h$-definable.

Suppose that $S'(x,y,z)$ holds. Then $x,y,z\in X_4$ and there is $w\in \Q$ satisfying
$$
|h(z)-h(w)|\le 4\cdot 48=192.
$$
Since $S(x,y,w)$ holds, we get
$$
|h(x)+h(y)-h(z)|\le 192,
$$
and we deduce
$$
\begin{aligned}
|\theta(x)&+\theta(y)-\theta(z)|\cdot D \\
&= |(\theta(x)D-h(x)) +(\theta(y)D-h(y))-(\theta(z)D-h(z)) +  (h(x)+h(y)-h(z))|\\
&\le  3\cdot 16+192=240.
\end{aligned}
$$
Since $\theta$ takes integer values and $D>30000$ we conclude $\theta(x)+\theta(y)=\theta(z)$.

Conversely, suppose that $x,y,z\in X_4$ satisfy $\theta(x)+\theta(y)=\theta(z)$. Then we have
$$
\begin{aligned}
|h(x)+h(y)-h(z)|&=|(h(x)-\theta(x)D)+(h(y)-\theta(y)D)-(h(z)-\theta(z)D)|\\
&\le  3\cdot 16=48.
\end{aligned}
$$
Let $u,v$ be positive integers with $h(x)=\log u$ and $h(y)=\log v$, and let $w=uv$. Then $S(x,y,w)$ holds and we deduce
$$
|h(w)-h(z)|\le 48,
$$
which implies that $E^{48}(w,z)$ holds. Hence, $S'(x,y,z)$ holds.
\end{proof}
Recall the binary relation $B\subseteq \N^2$ from Section \ref{SecAuxiliary}.
\begin{lemma}\label{LemmaDef3} The set $(\theta^2)^{-1}(B)\subseteq X_4^2$ is positive existentially $\Lcal_h$-definable in $\Q^2$.
\end{lemma}
\begin{proof} Recall the notation from Section \ref{SecEllChoice}. In particular, the point $Q_1$ generates the group $\Gamma$ and for every integer $k$ we have the point $Q_k=[k]Q_1\in \Gamma$. Observe that $X_1\subseteq X_4$ because $0\in X_1$.  We note that the set
$$
\Ccal = \{(\gamma,\delta)\in \Q^2 : \mbox{ there is $k\in \Z-\{-1,0\}$ with $\gamma=\pi(Q_k)$ and $\delta=\pi(Q_{k+1})$}\}\subseteq X_1^2\subseteq X_4^2
$$
is positive existentially $\Lcal_h$-definable. Indeed, note that $\Gamma-\{Q_0\}=\Gamma\cap \Q^2$ (the affine points of $\Gamma$) and the set  $\Ccal$ is defined by the following expression:
$$
\exists Q,R\in \Gamma\cap \Q^2, (R=Q+Q_1)\wedge (\gamma=\pi(Q))\wedge (\delta =\pi(R)).
$$
Using the set $\Ccal$ together with the relation $E_{1,1}(x,y)$ (and a disjunction to include the cases $k=-1,0$) we deduce that the set
$$
\Ccal' = \{(\gamma,\delta)\in \Q^2 : \mbox{ there is $k\in \Z$ with $h(\gamma)=h_\pi(Q_k)$ and $h(\delta)=h_\pi(Q_{k+1})$}\}\subseteq X_4^2
$$
is positive existentially $\Lcal_h$-definable. 

We claim that the set
$$
\Ccal'_+ = \{(\gamma,\delta)\in \Q^2 : \mbox{ there is $k\in \N$ with $h(\gamma)=h_\pi(Q_k)$ and $h(\delta)=h_\pi(Q_{k+1})$}\}\subseteq X_4^2
$$
is the same as the set
$$
\{(\gamma,\delta)\in \Ccal': L(\gamma,\delta)\}
$$
with $L(x,y)$ being the relation from Lemma \ref{LemmaLess}. Indeed, since $|h_\pi(Q_k)-k^2D|\le 4$ and $D>30000$, this follows from the observation that $k^2 < (k+1)^2$ if and only if $k\ge 0$.

Finally, we claim that $B(\theta(x),\theta(y))$ holds for a given pair $(x,y)\in X_4^2$ if and only if the following holds:
$$
B'(x,y):\quad (x,y\in X_4) \wedge \exists (\gamma, \delta)\in \Ccal'_+, E^{20}(x,\gamma)\wedge E^{20}(y,\delta).
$$
This suffices, as the previous relation is positive existentially $\Lcal_h$-definable.

First, suppose that $B(\theta(x),\theta(y))$ holds for a pair $(x,y)\in X_4^2$. Let $k\ge 0$ be such that $\theta(x)=k^2$ and $\theta(y)=(k+1)^2$. Then
$$
|h(x)-k^2D|\le 16\mbox{ and }|h(y)-(k+1)^2D|\le 16.
$$
We deduce that
$$
|h(x)-h_\pi(Q_k)|=|(h(x)-k^2D)+(\hat{h}(Q_k)-h_\pi(Q_k))|\le 16+4=20
$$
and similarly, $|h(y)-h_\pi(Q_{k+1})|\le 20$. We can take $\gamma=\pi(Q_k)$ and $\delta=\pi(Q_{k+1})$ (unless $k=0$ in which case we take $\gamma=0$) to conclude that $B'(x,y)$ holds.

Conversely, suppose that $B'(x,y)$ holds and let $(\gamma,\delta)\in \Ccal'_+$ be as in the definition of $B'(x,y)$. Let $k\ge 0$ be such that $h(\gamma)=h_\pi(Q_k)$ and $h(\delta)=h_\pi(Q_{k+1})$ and note that 
$$
|h(\gamma)-k^2D|=|h_\pi(Q_k)-\hat{h}(Q_k)|\le 4,
$$
and similarly $|h(\delta)-(k+1)^2D|\le 4$. Since $E^{20}(x,\gamma)$ holds, we deduce
$$
|h(x)-k^2D|=|h(x)-h_\pi(\gamma)+h_\pi(\gamma)-k^2D|\le 4\cdot 20 + 4=84
$$
concluding that $\theta(x)=k^2$. Similarly we get $\theta(y)=(k+1)^2$.
\end{proof}

\subsection{Proof of the main result} Finally, we can prove Theorem \ref{ThmMain1}:

\begin{proof}[Proof of Theorem \ref{ThmMain1}] By Lemma \ref{LemmaMult}, it suffices to give a positive existential interpretation of the $\Lcal_B$-structure $(\N;0,1,+,B,=)$ in the $\Lcal_h$-structure $\Qfrak$. This is achieved by the function $\theta:X_4\to \N$ which is surjective (cf. Lemma \ref{LemmaTheta}) and it satisfies the required definability conditions by Lemmas \ref{LemmaDefXn}, \ref{LemmaDef1}, \ref{LemmaDef2}, and \ref{LemmaDef3}.
\end{proof}

\section{Acknowledgments}

N.G.-F. was supported by ANID Fondecyt Regular grant 1211004 from Chile.

H.P. was supported by ANID Fondecyt Regular grant 1230507 from Chile.

X.V. was supported by ANID Fondecyt Regular grant 1210329 from Chile. 

This material is based upon work supported by the National Science Foundation under Grant No. DMS-1928930 while the authors participated in the program Definability, Decidability, and Computability in Number Theory, part 2,  hosted by the Mathematical Sciences Research Institute in Berkeley, California, during the Summer of 2022.

The authors are deeply indebted to Thanases Pheidas who, for many years, generously shared his ideas on Hilbert's tenth problem for $\Q$ and other open problems. In particular, this article would not exist without the key contributions of Thanases during our visit to MSRI in the Summer of 2022.   We respectfully dedicate this work to his memory.


\end{document}